\documentclass{amsart}
\usepackage{amssymb,latexsym}
\theoremstyle{plain}
\newtheorem{theorem}{Theorem}

\newtheorem{lemma}{Lemma}
\theoremstyle{definition}

\newtheorem{remark}{Remark}

\newtheorem{conjecture}{Conjecture}

\begin{document}

\title[nodal curves]
{Nodal curves and components of the Hilbert scheme of curves in $\mathbb {P}^r$ with the expected number of moduli}
\author{E. Ballico}
\address{Dept. of Mathematics\\
 University of Trento\\
38123 Povo (TN), Italy}
\email{ballico@science.unitn.it}
\thanks{The author was partially supported by MIUR and GNSAGA of INdAM (Italy).}

\begin{abstract}
We study the existence of components with the expected number of moduli of the
Hilbert scheme of integral nodal curves $C \subset \mathbb {P}^r$ with prescribed degree, arithmetic genus and number
of singular points.
\end{abstract}

\maketitle

\section{Introduction}\label{S1}

For all integers $g, r, d$ set $\rho (g,r,d):= (r+1)d -rg -r(r+1)$.  We work in the range
$\rho (g,r,d) <0$. Several authors studied the existence of irreducible components of the Hilbert scheme of smooth curves of $\mathbb {P}^r$ with the expected number of moduli (\cite{s1},\cite{eh1},\cite{eh2},\cite{be4},\cite{be},\cite{l1},\cite{p},\cite{l2},\cite{e},\cite{st}). In the set-up of irreducible components, $\Gamma$, of $\mbox{Hilb}(\mathbb {P}^r)$ with the expected dimension and the expected number of moduli in the range $\rho (g,r,d)\le 0$ it also says that for a general $C\in \Gamma$
the $g^r_d$ coming from the inclusion $C\subset \mathbb {P}^r$ is an isolated point of the set of all $g^r_d$'s on $C$ (see \cite{p}, Definition 1.1.2); it does not say that $C$ has only finitely many $g^r_d$'s. The result announced in \cite{eh1}, part 6) at page 338, is stronger for the following reason: it asks about the existence of an irreducible variety $T\subset \mathcal {M}_g$ over which the relative $\mathcal {G}^r_d$ is generically finite, but
$T\nsubseteq V$ for any integral $V\subseteq \mathcal {M}_g$ with $\dim (V) > \dim (T)$ and a general $C\in V$ has a $g^r_d$. As far as we know no proof of \cite{eh1}, part 6) at page 338, was published. Our tools give no informations on these two more general (and very interesting) problems.

For any nodal and connected curve $A\subset \mathbb {P}^r$ let $N_A$ denote its normal bundle in $\mathbb {P}^r$. If $h^1(A,N_A)=0$, then $\mbox{Hilb}(\mathbb {P}^r)$ is smooth
at $A$ and hence $\mbox{Hilb}(\mathbb {P}^r)$ has a unique irreducible component, $\Gamma$, and $\dim (\Gamma ) =h^0(A,N_A) = (r+1)\cdot \deg (A)-(r-3)(p_a(A)-1)$. In this case we say that
$\Gamma$ has the expected dimension. For all integers $r, d, g, t$ such that $0 \le t \le g$ let $E(r,d,g,t)$ denote the subset of the Hilbert scheme of $\mathbb {P}^r$ parametrizing integral
and non-degenerate curves $C\subset \mathbb {P}^r$ such that $\deg (C)=d$, $p_a(C) =g$ and $C$ has exactly $t$ ordinary nodes
as its only singularities. Set $E'(r,d,g,t):= \{C\in E(r,d,g,t): h^1(C,\mathcal {O}_C(2))=0\}$. We will always take the reduced structure as the scheme structure
for $E(r,d,g,t)$ and its open subset $E'(r,d,g,t)$. Let $\overline{E}'(r,d,g,t)$ denote the closure of $E'(r,d,g,t)$ in the Hilbert scheme of $\mathbb {P}^r$. Let $\overline{\mathcal {M}}_g\{t\}$ denote the set of all integral $Y\in \overline{\mathcal {M}}_g$
with exactly $t$ nodes and $\overline{\mathcal {M}}_g[t]$ its closure in $\overline{\mathcal {M}}_g$. Notice
that $\overline{\mathcal {M}}_g\{t\}$ is integral and $\dim( \overline{\mathcal {M}}_g\{t\} ) =3g-3-t$. For instance, $\overline{\mathcal {M}}_g\{g\}$ parametrizes the set of all rational curves with exactly $g$ nodes.
Notice that we always have a morphism $u_{r,d,g,t}: E(r,d,g,t) \to \overline{\mathcal {M}}_g\{t\}$. Let $\Gamma$ be any irreducible component of $E'(r,d,g,t)$. We say
that $\Gamma$ has the expected number of moduli if $\dim (u_{r,d,g,t}(\Gamma )) = \min \{3g-3-t,3g-3 +\rho (d,g,r)-t\}$.

\begin{conjecture}\label{i1}
Fix an integer $r \ge 3$. There is a function $\mathbb {O}_r: \mathbb {N} \to \mathbb {N}$ such that $\lim _{g\to +\infty} \mathbb {O}_r(g)/g = 0$ and the following property holds:
there is an integer $g_0$ such that for all integers $g, t, d$ with $g\ge g_0$, $0 \le t \le g$, $-\rho (g,r,d) +t \le 2g -\mathbb {O}_r(g)$ there is an irreducible
component $\Gamma$ of $E'(r,d,g,t)$ with the expected dimension and the expected number of moduli.
\end{conjecture}

\begin{conjecture}\label{i2}
Fix integers $r\ge 3$ and $t\ge 0$. There is a function $\mathbb {O}_{r,t}: \mathbb {N} \to \mathbb {N}$ such that $\lim _{g\to +\infty} \mathbb {O}_{r,t}(g)/g = 0$ and the following property holds:
there is an integer $g_0 \ge t$ such that for all integers $g, x, d$ with $g\ge g_0$, $0\le x \le t$, $-\rho (g,r,d) +x \le 3g -\mathbb {O}_{r,t}(g)$ there is an irreducible
component $\Gamma$ of $E'(r,d,g,x)$ with the expected dimension and the expected number of moduli.
\end{conjecture}

The case $t=0$ of Conjecture \ref{i2} is the question raised in \cite{l2}, after the statement of Theorem 1.2. This case was proved for $r=3$ and $t=0$ in \cite{p}.

In this paper we prove the following result.

\begin{theorem}\label{i3}
Fix integers $r\ge 4$ and $t\ge 0$. There is a function $\psi _{r,t}: \mathbb {N} \to \mathbb {N}$ such that $\lim _{g\to +\infty} \psi_{r,t}(g)/g = 0$ and the following property holds:
there is an integer $g_0 \ge t$ such that for all integers $g, x, d$ with $g\ge g_0$, $0\le x \le t$, $-\rho (g,r,d)\le 2rg/(r+3) -\psi_{r,t}(g)$ there is an irreducible
component $\Gamma$ of $E(r,d,g,x)$ with the expected number of moduli. If $r\ge 7$, then $\Gamma$ is a component of $E'(r,d,g,x)$.
\end{theorem}

Our proof of Theorem \ref{i3} is just a small modification of the proofs in \cite{l2}. We only write the modifications needed to get
nodal irreducible curves.

We work over an algebraically closed field $\mathbb {K}$ such that $\mbox{char}(\mathbb {K})=0$.

\section{The proof}\label{S2}

\begin{lemma}\label{a00}
Fix a set $S\subset \mathbb {P}^m$, $m \ge 2$, such that $\sharp (S) \le m+3$ and $S$ is in linearly general position. Let $\Gamma$ be the set of all rational normal curves of $\mathbb {P}^m$
containing $S$. Then $\Gamma$ is a non-empty and irreducible variety of dimension $(m+3-\sharp (S))(m-1)$.
\end{lemma}

\begin{proof}
If $\sharp (S)=m+3$, then the result is classical. Now assume $\sharp (S)\le m+2$. Fix a general $A\subset \mathbb {P}^r$ such that $\sharp (A)=s+3-\sharp (S)$. Since $\sharp (S\cup A) =m+3$,
we may apply the case just done to the set $S\cup A$. We get the nonemptiness and the irreducibility of $\Gamma$ and its dimension, because
for a fixed rational normal curve $D\subset \mathbb {P}^m$ such that $D\supset S$ the set of all $A$'s contained in $D$ has dimension $s+3-\sharp (S)$.
\end{proof}

\begin{lemma}\label{u1}
Let $Y = A\cup B \subset \mathbb {P}^r$ be a nodal curve such that $h^1(A,\mathcal {O}_A(2))=0$. Set $S:= A\cap B$ and see it as an effective Cartier divisor
of $B$. Assume $h^1(B,\mathcal {O}_B(2)(-S)) =0$. Then $h^1(Y,\mathcal {O}_Y(2))=0$.
\end{lemma}

\begin{proof}
Since $h^1(B,\mathcal {O}_Y(2)(-S)) =0$, we have $h^1(B,\mathcal {O}_B(2)) =0$ and the restriction map $H^0(B,\mathcal {O}_B(2)) \to H^0(S,\mathcal {O}_S(2))$
is surjective. Hence the lemma follows from the Mayer-Vietoris exact sequence
$$0 \to \mathcal {O}_Y(2) \to \mathcal {O}_A(2)\oplus \mathcal {O}_B(2)\to \mathcal {O}_S(2) \to 0$$
\end{proof}

\vspace{0.3cm}

\qquad {\emph {Proof of Theorem \ref{i3}.}} Until the last step we will only get a component of $E(r,d,g,x)$. In the statement of \cite{l2}, Theorem 1.2, there is  function which goes as $2rg/(r+3)$ if $r\notin \{4,6\}$, like $(30/19)g$ if $r=4$ and
like $(3/2)g$ if $r=6$. For simplicity we used the weaker upper bound $2rg/(r+3)$ even for $r\in \{4,6\}$. Fix
any integer $x$ such that $0\le x \le t$. By \cite{l2}, Theorem 1.2, we may assume $x\ne 0$. Fix any
pair $(u,v)$ such that the proof of \cite{l2}, Theorem 1.2, gives the existence of an irreducible component
$A_{u,v,r}$ of $\mbox{Hilb}(\mathbb {P}^r)$ with the expected number of moduli
and whose general element $Y$ is a smooth curve of degree $u$ and genus $v$ with $h^1(Y,N_Y)=0$. From
the range of degrees and genera covered in the proof of \cite{l2}, Theorem 1.2, it
follows that to prove Theorem \ref{i3} it is sufficient to find (for all $(u,v)$!)
an irreducible component $B$ of $\mbox{Hilb}(\mathbb {P}^r)$ and a maximal subfamily $B_x\subset B$ of integral nodal curves with exactly
$x$ nodes and with the following properties:
\begin{itemize}
\item[(i)] B has the expected dimension and the expected number of moduli
and a general $Y\in B$ is a smooth curve of degree $u+xr$ and genus $v+x(r+1)$ with $h^1(Y,N_Y)=0$;
\item[(ii)] $B_x \subset B$ is a maximal subfamily of integral nodal curves with exactly
$x$ nodes of $B$, i.e. $B_x\subset B$ and $B_x$ is an open subset of an irreducible component of
$E'(r,u+rx,v+(r+1)x,x)$;
\item[(iii)] $B_x$ has the expected number of moduli, i.e. $u_{r,u+rx,v+(r+1)x,x}\vert B_x$ is generically
finite.
\end{itemize}
Fix $x(r+2)$ general points of $Y$ and divide them into $x$ subsets $A_1,\dots ,A_x$ with $\sharp (A_i)=r+2$
for all $i$. Let $D_i\subset \mathbb {P}^r$ be a general rational normal curve containing $A_i$. For
general $A_1,\dots ,A_x$ and general $D_1,\dots ,D_x$ we get $D_i\cap D_j=\emptyset$ for all $i\ne j$,
$D_i\cap Y = A_i$ for all $i$ and that the curve $W:= Y\cup D_1\cup \cdots \cup D_x$ is nodal. Notice that
$W$ is connected, $\deg (W)=u+rx$ and $p_a(W)=v+(r+1)x$. By \cite{be4}, Lemma 2.3 (which uses \cite{s1} and \cite{hh}, Theorem 4.1 and Remark 4.1.1), the curve $W$ is smoothable and $h^1(W,N_W)=0$.
To prove Theorem \ref{i3} we may also restrict the previous proof to pairs $(u,v)$ with the additional condition $\rho (v,r,u) \le 0$. In this
range the proof of \cite{l2}, Theorem 1.2, gives $h^0(Y,T\mathbb {P}^r\vert Y) = (r+1)^2-1$. As in \cite{s1} (use of the multiplication map
$\mu _0(D)$) or \cite{l2}, property ($\gamma$) at page 3489, to get that the unique irreducible
component $\Gamma$ of $\mbox{Hilb}(\mathbb {P}^r)$ containing $Y\cup D_1$ has the right number of moduli and that $h^0(X,T\mathbb {P}^r\vert X)
= (r+1)^2-1$ for a general element $X$ of it, it is sufficient to prove $h^0(D_1,(T\mathbb {P}^r\vert D_1)(-A_1))=0$. This vanishing is true, because
the vector bundle $T\mathbb {P}^r\vert D_1$ is a direct sum of $r$ line bundles of degree $r+1$ (\cite{v}, \cite{ra}, \cite{r0}). Iterating $x-1$ times the proof
we get that the irreducible component of $\mbox{Hilb}(\mathbb {P}^r)$ containing $W$ has the right number of moduli. Varying the curve $Y$ in $B$, the sets $A_i$ and the rational normal curves $D_i$
we get an irreducible family $\mathcal {W}$ of nodal curves of degree $u+rx$, arithmetic genus $v+x(r+1)$, $x+1$ irreducible components and with exactly $x(r+2)$ nodes. Since
$\dim (B) = (r+1)u-(r-3)(v-1)$, we have $\dim (\mathcal {W}) = (r+1)u -(r-3)(v-1) +x(r+2) +x(r-1)$ (use the case $m=r$ of Lemma \ref{a00}).  Since $\dim (\Gamma )= (r+1)(u+rx) -(r-3)(v+(r+1)x)) =(r+1)u-(r-3)(v-1) +(3r+3)x$,
$\mathcal {W}$ has codimension $x(r+2)$ in $\Gamma$. Fix any integer $y$ such that $0\le y \le x(r+2)$ and $S\subseteq A_1\cup \cdots \cup A_x$ such
that $\sharp (S)=y$. By \cite{s2}, Theorem 6.3, there is a neighborhood $U$ of $W$ in $\mbox{Hilb}(\mathbb {P}^r)$ and a non-empty locally closed subset $U_S$ of $U$ consisting of curves
with exactly $y$ nodes, each of them being a deformation of a different point of $S$, and $\dim (U_S)=\dim (U)-y$. Taking $y=x$ and $\sharp (S\cap A_i)=1$ for all $i$ we get
that any $E\in U_S$ is irreducible. To prove that $U_S$ has the expected number of moduli, it is sufficient to prove that $u_{g,r,d}\vert U_S$ is generically finite. Since $W\in \overline{U_S}$,
it is sufficient to prove that $\mathcal {O}_W(1)$ is an isolated $g^r_{\deg (W)}$ on $W$. Take a general $g^r_{\deg (W)}$, $L$, of an irreducible component $\Lambda$ of the set of all $g^r_{\deg (W)}$'s on $W$ such that $\mathcal {O}_W(1)
\in \Lambda$.
Since $\mathcal {O}_W(1)$ is very ample and $h^0(W,\mathcal {O}_W(1)) =r+1$, $L$ is very ample and $h^0(W,L)=r+1$. Call $h_L: W \hookrightarrow \mathbb {P}^r$ the embedding induced by the complete linear system $\vert L\vert$.
Since $B$ has the expected number of moduli, $Y$ is general in $B$ and $\rho (v,r,u) \le 0$, $\mathcal {O}_Y(1)$
is an isolated $g^r_u$ on $Y$.  Since $L\vert Y$ is near to $\mathcal {O}_Y(1)$, we have $L\vert Y \cong \mathcal {O}_Y(1)$. Since we may assume that $A_1\cup \cdots \cup A_x$
is not sent into itself by a non-trivial automorphism of $Y$, we also see that $h_L(W)$ is the union of $Y$ and $x$ rational normal curves $C_1,\dots ,C_x$ with $C_i\cap Y =A_i$ for all $i$
and $C_i\cap C_j = \emptyset$ for all $i\ne j$. Since $\sharp (A_i)\ge 4$, only finitely many automorphisms of $D_i$ send $A_i$ into itself.
Hence we only have finitely many possible curves $C_i$, $i=1,\dots ,x$. Hence $L = \mathcal {O}_W(1)$.

To check that the component we got is a component of $E'(r,d,g,x)$ we need to check that at each step here and in \cite{l2} we may apply Lemma \ref{u1}. For the curves $D_i$ it is easy. In \cite{l2}, page 3490, the author quoted
\cite{l1}, Sublemma 3.5; here $B$ is a smooth rational curve, $\deg (B)=r-1$ and $\sharp (S) =r+2$. In \cite{l2}, Claim (3.4) at page 3490, twice it is quoted \cite{l1}, Claim 3.7; here $B$ is a smooth elliptic curve,
$\deg (B)=r+1$. In \cite{l1}, Proposition 2.1, we have as $B$ a smooth curve of degree $d\ge p_a(B)+r$ with $\sharp (S) =r+4$; here $\deg (\mathcal {O}_B(2)(-S)) \ge 2\cdot p_a(B) +2r-r-4 > 2\cdot p_a(B)-2$. In \cite{l2} the reader will often meet
as $B$ a canonically embedded curve (hence $B$ has genus $r+1$ and $\mathcal {O}_B(1) \cong \omega _B$) with $\sharp (S) =r+6+ \epsilon $
with $\epsilon = 0$ if $r\notin \{4,6\}$, $\epsilon = 2$ if $r=4$ and $\epsilon =1$ if $r=6$; here we need $r\ge 7$ to use Lemma \ref{u1}. \qed

\begin{remark}\label{b1}
Fix $r\in \{3,4,6\}$. Using the cases $r\in \{4,6\}$ of \cite{l2}, Theorem 1.2,  or \cite{p} (case $r=3$) we may take
$3g$ (case $r=3$) or $(30/19)g$ (case  $r=4$) and $(3/2)g$ (case $r=6$) instead of $2rg/(r+3)$ in the statement of Theorem \ref{i3}.
\end{remark}

\providecommand{\bysame}{\leavevmode\hbox to3em{\hrulefill}\thinspace}

\end{document}